\documentclass{article}

\usepackage{amsmath}
\usepackage{amssymb}
\usepackage{amsthm}
\usepackage{enumerate}
\usepackage{bbm}

\usepackage{multicol}
\usepackage{amsfonts}
\usepackage[colorlinks=true,citecolor=red,linkcolor=blue]{hyperref}

\newtheorem{theorem}{Theorem}[section]
\newtheorem{proposition}[theorem]{Proposition}
\newtheorem{corollary}[theorem]{Corollary}
\newtheorem{lemma}[theorem]{Lemma}
\newtheorem{claim}[theorem]{Claim}

\newtheorem{rmk}[theorem]{Remark}
\newtheorem{question}[theorem]{Question}

\newtheorem*{theorem*}{Theorem}
\newtheorem*{corollary*}{Corollary}

\newtheorem{definition}[theorem]{Definition}

\newcommand{\T}{\mathcal{T}}
\newcommand{\F}{\mathcal{F}}

\newcommand{\A}{\mathbb{A}}

\newcommand \hyp {\operatorname{Exp}}

\renewcommand{\hyp}{\operatorname{Exp}}

\providecommand{\keywords}[1]
{
	\small	
	\textbf{\textit{Keywords:}} #1
}
\providecommand{\subjclass}[1]
{
	\small	
	\textbf{\textit{MSC2020:}} #1
}
\title{AUTOHOMEOMORPHISMS OF THE FINITE POWERS OF THE DOUBLE ARROW.}
\author{ Sebasti\'an Barr\'ia, Carlos Martinez-Ranero \footnote{Universidad de Concepci\'on, Chile (all authors).} }

\begin{document}
	\maketitle
	
	\begin{abstract}
		Let $\mathbb{A}$ and $\mathbb{S}$ denote the double arrow of Alexandroff and the Sorgenfrey line, respectively.
		We show that   any homeomorphism $h:^m\mathbb{A}\to^m\mathbb{A} $ is locally (outside of a nowhere dense set) a product of monotone embeddings $h_i:J_i\subseteq \mathbb{A}\to\mathbb{A} (i\in m)$ followed by a permutation of the coordinates.\\
		We also prove that the symmetric products  $\mathcal{F}_m(\mathbb{A})$ are not homogeneous for any $m\geq 2$.  This partially solves an open question of A. Arhangel'ski\v{i} \cite{ar87}. In contrast, we show that symmetric product $\mathcal{F}_2(\mathbb{S})$ is homogeneous. 
		\footnote{The first named author was partially supported by CONICYT-PFCHA Doctorado Nacional 2017-21170738. The  second named author was partially supported by  Proyecto VRID-Investigación  No. 220.015.024-INV }
	\end{abstract}
	
	
	\
	\tableofcontents
	
	\section{INTRODUCTION}
	A space $X$ is \textit{homogeneous} if for every $x,y\in X$ there exists a autohomeomorphism $h$ of $X$ such that $h(x)=y$. Several classic results on homogeneity involve the study of the hyperspace $Exp(X)$ (set of closed subsets  of $X$) in the Vietoris topology. 
	In this paper we are motivated by the following general question.
	\begin{question}
		When is $\hyp(X)$ homogeneous?
	\end{question}
	
	In the seventies, it was shown by Shori and West \cite{sw} that $\hyp([0,1])$ is homeomorphic to the Hilbert's cube. In particular, is possible that the $\hyp(X)$ is homogeneous while $X$ is not. On the other hand, if  $\kappa>\aleph_1$, then $\hyp{(2^{\kappa})}$ is not homogeneous (see \cite{sce}). Thus, the question of homogeneity of the hyperspace turns out to be quite subtle. 
	
	Arhangel'ski\v{i} in \cite{ar87}  asked the following question (it also appears in \cite{avm}).
	\begin{question}\label{quesarhan}
		Is the hyperspace  $\hyp{(\A)}$  homogeneous?
	\end{question} 
	
	In this paper we partially answer Question \ref{quesarhan}, by showing that.
	
	\begin{theorem}\label{main2}
		The symmetric product  $\mathcal{F}_m(\mathbb{A})$ is not homogeneous for any $m\geq 2$. 
	\end{theorem}
	
	Where the symmetric product $\mathcal{F}_m(\mathbb{A})$ is the subspace of $\hyp(\mathbb{A})$ consisting of all finite, non-empty, subsets of cardinality at most $m$. 
	
	The following result can be seen as a companion of the previous Theorem. 
	
	\begin{theorem}\label{main3}
		The symmetric product $\mathcal{F}_2(\mathbb{S})$ is homeomorphic to $\mathbb{S}^2$. In particular, it is homogeneous.
	\end{theorem}
	
	In the course of proving Theorem \ref{main2}, we study the group of autohomeomorphisms of $^m\mathbb{A}$, and obtain the following Theorem which give us a complete picture on the structure of such autohomeomorphisms.   
	
	\begin{theorem}\label{main1}
		Let $h:^m\mathbb{A}\to ^m\mathbb{A}$ be a homeomorphism. Then there is  a pairwise disjoint sequence of basic clopen boxes $U_n:=\prod\limits_{j\in m}I^j_n (n\in \omega) $ such that $\bigcup U_n$ is dense in $^m\mathbb{A}$ and $h\restriction U_n= \sigma\circ h^0\times\cdots\times h^{m-1}, $ where each $h^j: I_n^j\to \mathbb{A}$ is an strictly monotone homeomorphism onto a clopen interval, and $\sigma$ is a permutation of $^m\mathbb{A}$.
	\end{theorem}

	The paper is organized as follows. In Section 2 we study the autohomeomorphisms of $^m\A$ and give a proof of Theorem \ref{main1}. In Section 3 we give a proof of Theorem \ref{main2} and in Section 4 a proof of Theorem \ref{main3}.  The notation and terminology in this paper is fairly standard.  We will use \cite{en89} as a basic reference on topology and \cite{avm} as a reference for  homogeneity and hyperspaces.

	\section{AUTOHOMEOMORPHISMS OF $^m\mathbb{A}$}
	The purpose of this section it to  prove Theorem \ref{main1}. It will be convenient to introduce some notation.
	Let $\mathbb{A}_0=]0,1]\times \{0\}, \mathbb{A}_1=[0,1[\times \{1\}$ and $\mathbb{A}=\mathbb{A}_0\cup \mathbb{A}_1$. Define the lexicographical ordering $\langle a,r\rangle \prec \langle b, s\rangle$ if $a<b$ or $a=b $ and $r<s$. The set $\A$ with the order topology is the \textit{double arrow space}.
	
	The  Let $\pi: \mathbb{A}\to [0,1]$ be the projection onto the first factor $\pi(\langle x,r\rangle )=x$, we will think of an element of the finite power $x\in ^m\mathbb{A}$ as function $x:m\to \mathbb{A}$. For any $a\in \mathbb{A}$ we will denote by $\overline{a}$ the constant sequence $a$ of arbitrary finite length.    Let $\pi_i:^m\mathbb{A}\to \mathbb{A}$ be the projection onto the $i$-coordinate, and let $h_i=\pi_i\circ h$. 
	Recall that a partial function $f:\mathbb{A}\to \mathbb{A}$ is monotone if it is either non-decreasing or non-increasing, and  $f$ is strictly monotone if it is either strictly increasing or strictly decreasing.

	We now recall the following result from R. Hernández-Gutiérrez.
	\begin{proposition}[\cite{hg} Proposition 3.1]\label{basicprop} 
		Let $h:\mathbb{A}\to\mathbb{A}$ be a continuous function, then there exists a pairwise disjoint  sequence $J_n (n\in\omega)$ of clopen intervals such that $\bigcup_{n\in\omega}J_n$ is dense in $\mathbb{A}$, and $h\restriction J_n$ is monotone, for any $n\in\omega.$
	\end{proposition}
	The following proposition tell us how continuous monotone functions look like locally. 
	\begin{proposition}\label{basiclema}
		Let $h:\mathbb{A}\to\mathbb{A}$ be a monotone continuous function. Then there is a clopen interval $J$ so that either $h\restriction J$ is constant or $h\restriction J$ is strictly monotone.
	\end{proposition}
	\begin{proof}
		On one hand, if there is a clopen interval $J$ so that $h\restriction J$ is an injection, then there is nothing to prove. On the other hand, if there are $x,y\in \mathbb{A} $ so that $\pi(x)\ne \pi(y)$ and $h(x)=h(y)$, then we are similarly  finished. If neither of the above alternatives hold, then there exists a dense set $D\subset \mathbb{A}\times \{\langle 0,1\rangle \}$ so that $h\restriction (\mathbb{A}\setminus D)$ is strictly monotone and $h(\langle a,0\rangle)=h(\langle a,1\rangle)$ for any $\langle a,0\rangle\in\mathbb{A}.$ However, such a function cannot be continuous. This finishes the proof of the Lemma.
	\end{proof}
	\begin{definition}
		Let $h:\mathbb{A}\to ^m\mathbb{A}$ and $j_0\in m$ be given. We say that a clopen interval $J$ is $j_0$-good for $h$	if   $ h_{j_0}\restriction J$ is  strictly monotone  and  $h_{j}\restriction J$ is constant for any $j\in m\setminus\{j_0\}$. We say that $J$ is good for $h$ if it is $j_0$-good for some $j_0\in m$.

	\end{definition}
	\begin{rmk}
		In other words, $J$ is $j_0$-good for $h$ if and only if $h$ sends $J$ into a line parallel to the $j_0$-th axis. 
	\end{rmk}
	
	The following lemma gives an indication as to why this definition will play a role.
	It will be used in the verification of Theorem \ref{main1}.
	
	\begin{lemma}\label{gdlema}
		Let $h:\mathbb{A}\to ^m\mathbb{A}$ be an embedding such that $h''[\mathbb{A}]$ is $G_\delta$ in $^m\mathbb{A}$. Then there exists  a pairwise disjoint sequence   $J_n (n\in\omega)$ of clopen intervals such  that $\bigcup_{n\in\omega} J_n$ is dense in $\mathbb{A}$ and for each $n,$ there is $j\in m$ so that  $J_n$ is $j$-good for $h$. 
	\end{lemma}
	\begin{proof}
		
		Let $U$ denote the union of all clopen intervals which are good for $h$. Since $\mathbb{A}$ is separable, it suffices to show that $U$ is dense. In order to get a contradiction, suppose that there is a nonempty  clopen interval $J$ disjoint from $U$. By going to a clopen sub-interval of $J$ if necessary, and applying, Proposition \ref{basicprop} and Lemma \ref{basiclema}, $m$ times, we may assume that $h_j\restriction J$ is either one-to-one
		or constant, for any $j\in m$. Since $h$ is an embedding there is at least one $j\in m,$ such that $h_j$ is non-constant (equivalently, strictly monotone). Therefore, we are left to show that there is at most one  $j\in m$ so that $h_j$ is non-constant.
		\begin{claim}
			If there exists $j_0\ne j_1\in m$ such that $h_{j_0}\restriction J$ and $h_{i_1}\restriction J$ are one-to-one, then  $X=h"[J]$ is not  a $G_\delta$ in $^m\mathbb{A}$.
		\end{claim} 
		\begin{proof}
			
			Let $X\subseteq\bigcap_{n\in \omega}U_n,$ where each $U_n$ is an open set.  Since $X$ is compact, we may assume, that $U_n=\bigcup_{i\in k_n } \prod\limits_{j\in m} I_{n,i}^j,$ where $I_{n,i}^j$ 
			are clopen intervals. Let $A$ be equal to 
			$$\{\pi(x):\exists n\in\omega \exists i\leq k_n, j\in\{j_0,j_1\} (x\in\{\min(I_{n,i}^j),\max(I_{n,i}^j)\})\}$$
			Fix $x\in X$ so that  $\pi(x(i_0))$ does not  belong to $A$, this is possible as $A$ is countable and $h_{i_0}\restriction J$ is an injection. We claim that both  points $$x\restriction (m\setminus\{i_0\})\cup \{(i_0,\langle \pi(x),0\rangle)\}\  \rm{ and}\ x\restriction (m\setminus\{i_0\})\cup \{(i_0,\langle \pi(x),1\rangle)\} $$  belong to $ \bigcap_{n\in \omega}U_n.$ However,  only one of them belongs to $X$ as $h_{j_1}$ is injective, which is a contradiction. In order to prove the claim, fix $N\in\omega.$ As $x\in X\subseteq \bigcap_{n\in \omega}U_n.$ We can find $i\in k_n$ so that $x\in  \prod\limits_{j\in m} I_{N,i}^j$. Note that, $\langle \pi(x(i_0)),0\rangle $ and  $\langle \pi(x(i_0)),1\rangle,  $ both belong to $I_{N,i}^{j_0}$, as neither of them are the maximum nor the minimum.  Thus, $x\restriction (m\setminus\{i_0\})\cup \{(i_0,\langle \pi(x),0\rangle)\}$   and $ x\restriction (m\setminus\{i_0\})\cup \{(i_0,\langle \pi(x),1\rangle)\}$ belong to $U_N$ as required. This finishes the proof of the Claim. \end{proof}
		
		Observe that since $h''[\mathbb{A}]$ is a $G_\delta $ in $^m\mathbb{A}$ and $J$ is clopen in $\mathbb{A}$, it follows that $X$ is also a $G_\delta $ in $^m\mathbb{A}$, which contradicts the previous
		Claim.

	\end{proof}
	
	We are now ready to prove the main Theorem of the section. 
	
	\begin{theorem}
		Let $h:^m\mathbb{A}\to ^m\mathbb{A}$ be a homeomorphism. Then there is  a pairwise disjoint sequence of basic clopen boxes $U_n:=\prod\limits_{j\in m}I^j_n (n\in \omega) $ such that $\bigcup U_n$ is dense in $^m\mathbb{A}$ and $h\restriction U_n= \sigma\circ h^0\times\cdots\times h^{m-1}, $ where each $h^j: I_n^j\to \mathbb{A}$ is an strictly monotone homeomorphism onto a clopen interval, and $\sigma$ is a permutation of $^m\mathbb{A}$.
	\end{theorem}
	\begin{proof}
		Since $^m\mathbb{A}$ is separable, and every clopen box is homeomorphic to $^m\mathbb{A}$ via a product of strictly increasing homeomorphisms.  Thus, it suffices to show that there is a clopen box so that  $h$ restricted to it  is as desired.    For each  $a\in ^{m-1}\mathbb{A}$ and $i\in m,$ define the line $E_{a,i}=\{ x\in ^m\mathbb{A}: x\restriction (m\setminus\{i\})=a\}$, and define an embedding $e_{a,i}:\mathbb{A}\to ^m\mathbb{A}$  by $e_{a,i}(p)=a\cup \{(i,p)\}$. Note that $h''[E_{a,i}]$ is  a $G_\delta$ in  $^m\mathbb{A},$ since $E_{a,i}$ is $G_\delta$ in $^m\mathbb{A}$, and $h$ is a homeomorphism. 
		We will recursively construct, for $j\in m$, clopen boxes $V_j:=\prod\limits_{i\in m} J_i^j\subset ^{m}\mathbb{A} $ and functions  $\sigma_j:\{0,\dots,j\}\to m$ such that
		\begin{enumerate}[i.]
			
			\item  $V_{j+1}\subset V_j$ 
			for $j\in m-1.$
			\item For each $i<j$ and $a\in \pi_{m-\{i\}}''[V_j]$, $ I^j_i$ is $\sigma(i)$-good for $h\circ e_{a,i}$. 
			\item $\sigma_{j+1}\restriction j=\sigma_j$ and $\sigma_i$ is injective for $j\in m-1.$
		\end{enumerate}
		Suppose we have constructed $V_j, \sigma_j$ for $j<k\leq m$. By applying Lemma \ref{gdlema} to the map $h\circ e_{a,k}\restriction I^{k-1}_i$, we can find for each $a\in A:=\pi_{m-\{k\}}''[V_{k-1}],$ rationals   $q_a,r_a\in\mathbb{Q}$ and an integer $j_{a,k}$ such that $[\langle q_a,1\rangle, \langle r_a,0\rangle])\subset I^{k-1}_i $ is $j_{a,k}$-good for $h\circ e_{a,k}$. Since $A$ is a Baire space, there exists $j_0,q,r$ such that $A_{j_0,q,r}:=\{x\in A:q_x=q, r_x=r, j_{x,0}=j_0\} $ is dense in some clopen  box $V:=J_0\times\cdots\times J_{m-1}\subseteq A$.

		\begin{claim}

			The interval $[\langle q,1\rangle, \langle r,0\rangle] $ is $j_0$-good, for $h\circ e_{x,k}$ for any $x\in V$.
		\end{claim}
		
		\begin{proof}
			
			We must show that $h_j\circ e_{x,k}\restriction [\langle q,1\rangle, \langle r,0\rangle]$ is constant for any $j\in m\setminus\{j_0\}$. In order to do so, pick $j\in m\setminus\{j_0\}, x\in V$ and $s,t\in[\langle q,1\rangle, \langle r,0\rangle]$. Choose a sequence $x_n $ of elements of $A_{j_0,q,r}$ converging to $x$, this is possible as  $A_{j_0,q,r}$ is dense in $V$.  Notice that $e_{x_n,k}(s)$ and $e_{x_n,k}(t)$  converge to $e_{x,k}(s)$ and $e_{x,k}(t)$, respectively. By assumption, we have that  $h_i(e_{x,k}(s))=\lim\limits_{n\to\infty}h_i(e_{x_n,k}(s))=\lim\limits_{n\to\infty}h_i(e_{x_n,k}(t))=h_i(e_{x,k}(t)).$  \end{proof}
		
		Define $$V_k=J_0\times\dots\times J_{k-1}\times[\langle q,1\rangle, \langle r,0\rangle]\times J_{k+1}\times\dots\times J_{m-1} \ \rm{ and} \  \sigma_k=\sigma_{k-1}\cup\{(k,j_0)\}.$$ It follows from the previous claim that properties  i. and ii. hold and also clearly  $\sigma_k$ extends $\sigma_{k-1}$. Hence, we are left to show that $\sigma_k$ is injective. Suppose for a contradiction that $\sigma_k(i_0)=\sigma_k(i_1)$ for some $i_0\ne i_1$. Pick $a\ne b\in V_k$ so that $$a^{i_0}:=a\restriction m-\{i_0\}=b\restriction m-\{i_0\}=:b^{i_0},\ \rm{ and\ let}\  a^{i_1}=a\restriction m-\{i_1\}, b^{i_1}=b\restriction m-\{i_1\}.$$  Notice that $|E_{a^{i_0},i_0}\cap E_{a^{i_1},i_1}|=1. $ Hence, it follows that $h''[E_{a^{i_0},i_0}\cap V_k]\subset E_{h(a)^{i_0},\sigma(i_0)}$ and also $h''[E_{a^{i_1},i_1}\cap V_k]\subset E_{h(a)^{i_1},\sigma(i_1)}.$ By assumption, $E_{h(a)^{i_0},\sigma(i_0)}=E_{h(a)^{i_1},\sigma(i_1)}.$ However, this would imply that $h''[E_{a^{i_1},i_1}\cap V_k]\cap h''[E_{b^{i_0},i_0}\cap V_k]=\emptyset$, which contradicts the fact that $|E_{a^{i_1},i_1}\cap E_{b^{i_0},i_0}\cap V_k|=1.$
		Finally, let $\sigma=\sigma_{m-1}$ and let $h^j=h_j\circ e_{a^j,j}\restriction I^k_j$ for some fixed $a\in V_k$ and $j\in m$. It follows from our construction 
		that $h\restriction V_k=\sigma\circ h^0\times\dots\times h^m$ as required. \end{proof}
	
	The previous Theorem  give us an a posteriori explanation of why the space $^m\mathbb{A}$ is not countable dense homogeneous for any $m\geq 2$. A fact first observed by Arhangel'ski\v{i} and van Mill \cite{avm} for $m=1$ and by Hernández-Gutiérrez in the case $m\geq 2.$ 
	
	\begin{corollary}[\cite{hg}]
		The space $^m\mathbb{A}$ is not countable dense homogeneous for any $m\geq 1.$
		
	\end{corollary}

	\section{NON-HOMOGENEITY OF $\mathcal{F}_m(\mathbb{A})$}
	
	In this section we prove Theorem \ref{main2}.  It will be convenient to introduce some notation.
	Let $\Delta_m=\{x\in ^m\mathbb{A} : \forall i\in m-1 (x(i)\leq x(i+1))\}$, let $\rho:\Delta_m\to \mathcal{F}_m(\mathbb{A})$ be the map given by $\rho(x)=\{x(0),\dots,x(m-1)\}$ and let $\sim $ denote the equivalence relation on $\Delta_m$ defined by $x\sim y$ if and only if $\rho(x)=\rho(y).$ Finally,  let $q:\Delta_m\to \Delta_m/\sim$ be the quotient map, we will sometimes write $[x]$, instead of $q(x),
	$ to represent the equivalence class. We consider $\Delta_m/\sim$ as a topological space with the quotient topology.

	The following classical fact give us a more geometric representation of $\mathcal{F}_m(\mathbb{A})$.
	\begin{proposition}[\cite{g54}]
		\label{ahom}
		The  map $\Tilde{\rho}: \Delta_m /\sim\to \mathcal{F}_m(\mathbb{\A})$ given by $\Tilde{\rho}([x])=\rho(x)$ is a homeomorphism.
	\end{proposition} 
	The next proposition is straightforward and it is left to the reader.
	\begin{proposition}
		Every clopen subset of $^m\mathbb{A}$ is homeomorphic to $^m\mathbb{A}$.
		
	\end{proposition}

		
		
			
		

	\begin{lemma}
		If $\F_m(\A)$ is homogeneous, then it is homeomorphic to $^m\mathbb{A}$.
		
	\end{lemma}
	\begin{proof}
		Suppose $\F_m(\A)$ is homogeneous, then there is an autohomeomorphism $h:\Delta_m/\sim \to \Delta_m/\sim$  such that $h([\overline{0}])=[x],$ where $x_0<x_1<\dots<x_{m-1}$.  On one hand, notice that, if  $J_0<\dots<J_{m-1}$ is a sequence of pairwise disjoint clopen intervals with $x_i\in J_i$ for $i\in m,$ then $q\restriction \prod\limits_{i\in m} J_i:\prod\limits_{i\in m}J_i\to \Delta_m/\sim $ is an homeomorphism. On the other hand, observe that for any $0<\epsilon<1$ the clopen cube $^m[0,\epsilon[$ is a saturated neighborhood of $\overline{0} $ such that $q''(^m[0,\epsilon[)$ is homeomorphic to $\Delta_m/\sim.$ Since $h$ is continuous there is an $\epsilon>$ such that $h''(^m[0,\epsilon[/\sim )\subset \prod\limits_{i\in m} J_i.$ It follows, from the previous Lemma, that $^m\mathbb{A}\cong \prod\limits_{i\in m} J_i\cong h''(^m[0,\epsilon[/\sim )\cong ^m[0,\epsilon[/\sim\cong \Delta_m/\sim.$ 
		
	\end{proof}
	\begin{theorem}
		The hyperspace $\mathcal{F}_m(\mathbb{A})$ is not homogeneous for any $m\geq 2$. 
	\end{theorem}
	\begin{proof}
		Suppose for a contradiction that there is a homeomorphism $h: \Delta_m/\sim \to ^m\mathbb{A},$ and let $\Gamma=\{[\overline{x}]\in \Delta_m/\sim : x\in\mathbb{A}\}.  $
		Observe that $\Gamma$ is homeomorphic to $\mathbb{A}$ and it is not $G_\delta$ in  $\Delta_m/\sim$ as $q^{-1}(\Gamma)=\{(x,\dots,x)\in ^m\mathbb{A}: x\in \mathbb{A\}}$ is not a $G_\delta$ in $\Delta_m$ (by Lemma \ref{gdlema}). We now  consider the embedding $\alpha: \mathbb{A}\to ^m\mathbb{A}$ given by $\alpha(x)=h([\overline{x}])$. Since $h''[\Gamma]$ is not a $G_\delta$ in $^m\mathbb{A},$ it follows, again by Lemma \ref{gdlema}, that there exits $j_0\ne j_1\in m$ and a clopen interval $J$ such that $\alpha_{i_0}:=\pi_{i_0}\circ \alpha $ and $\alpha_{i_1}:=\pi_{i_1}\circ \alpha$ are strictly monotone restricted to $J$. We will assume that both $\alpha_{i_0}\restriction J, \alpha_{i_1}\restriction J$ are strictly increasing, as the other cases are analogous.  
		\begin{claim}
			There is a countable subset $C\subset \pi''[J]$ such that $\pi(\alpha_{i_0}(\langle a,0\rangle))=\pi(\alpha_{i_0}(\langle a,1\rangle))$ and $\pi(\alpha_{i_1}(\langle a,0\rangle))=\pi(\alpha_{i_1}(\langle a,1\rangle))$ for any $a\in \pi''[J]\setminus C.$
		\end{claim}
		\begin{proof}
			Let $C_k=\{a\in \pi''[J]: \alpha_{i_k}(\langle a,0\rangle )< \pi(\alpha_{i_0}(\langle a,1\rangle )]\}$ for $k\in 2$. For each $a\in C_k,$ pick a rational $q_a$ such that $\alpha_{i_k}(\langle a,0\rangle )<q_a< \pi(\alpha_{i_0}(\langle a,1\rangle )$. Observe that since $\alpha_{i_k}$ is strictly monotone, the map $f:C_k\to \mathbb{Q}$ given by $f(a)=q_a,$ is one-to-one. Thus, $C=C_0\cup C_1$ is countable as desired.  
		\end{proof}
		For each $a \in A:=\pi''[J]\setminus C.$ Let $P_a^-=\alpha(\langle a,0\rangle), Q_a^+=\alpha(\langle a,1\rangle)$ and let 
		$$P_a^+=\alpha(\langle a,0\rangle)\restriction_{(m-\{i_0\})}\cup (i_0,\langle \pi(\alpha_{i_0}(\langle a,0\rangle ),1\rangle)$$
		and 
		$$Q_a^-=\alpha(\langle a,1\rangle)\restriction_{(m-\{i_1\}}\cup (i_1,\langle \pi(\alpha_{i_1}(\langle a,0\rangle ),1\rangle).$$
		
		Pick an element $[x_a]$ belonging to $$h^{-1}(\{P^-_a,P^+_a,Q^-_a,Q^+_a\})\setminus \{\{\Tilde{\rho}^{-1}(\langle a,0\rangle\}),\Tilde{\rho}^{-1}( \{\langle a,1\rangle\}),\Tilde{\rho}^{-1}(\{\langle a,0\rangle,\langle a,1\rangle\})  \}. $$ 
		Observe that, by our choice of $x_a,$ for any $x\sim x_a$ there is a $j\in m $ so that $\pi(x(j))\ne a.$ By successively refining $A$, we can find an uncountable subset $B\subset A,$  a natural number $N$ so that $q^{-1}([x_a])=\{x^i_a: i\in N\},$ an $N$-tuple $(j_0,\dots,j_{N-1})\in m^N$, a rational number $ q\in \mathbb{Q}$   such that $x_a^i(j_i)=x_a^{i'}(j_{i'})$ for all $i,i'\in N$ and $q\in]\pi(x^i_a(j_i)),a[$ \footnote{We are using the convention that $]a,b[=]\min(a,b),\max(a,b)[.$   } and  $h([x_a])=P^+_a$ (the case $h([x_a])=Q^-_a$ is analogous),  for any $a\in B.$
		Consider the clopen cubes   $U:=^m[\langle q,1\rangle, \langle 1,0\rangle ]$ and $V:=^m[\langle 0,1\rangle, \langle q,0\rangle ]$. Since both cubes are saturated we have that $\Tilde{U}:=q''[U]$ and $\Tilde{V}:=q''[V]$ form a clopen partition of $\Delta_m/\sim.$  Observe that $X:=\{[\overline{\langle a, 0\rangle}] : a\in B\}\subset \Tilde{U} $ and $Y:=\{[x_a]: a\in B\}\subset \Tilde{V}$, since $B$ is infinite (uncountable) and $\Delta_m/\sim$ is compact, it follows that its accumulation points are non-empty and disjoint. However, this contradicts the fact that $h''[X]=\{P_a^-: a\in B\}$ and $h''[Y]=\{P_a^+: a\in B\}$ have the same set of accumulation points. This finishes the proof of the Theorem. 
		
	\end{proof}
	
	It would be interesting to see if the above theorem can be extended to the hyperspace of all non-empty finite subsets $\mathcal{F}(\mathbb{A})$.
	\begin{question}
		Is the hyperspace $\mathcal{F}(\mathbb{A})$ homogeneous?
	\end{question}

	\section{HOMOGENEITY OF $\mathcal{F}_2(\mathbb{S})$.}
	
	In this section we prove Theorem \ref{main3}. It is worth mentioning that our proof is based on work of Bennett, Burke and Lutzer \cite{bbl} and we will also borrow some of its notation. 
	
	First of all, since the Sorgenfrey line is homeomorphic to $[0,1[$ with the subspace topology, we will assume that the $\mathbb{S}=[0,1[$. A  \textit{Sorgenfrey rectangle} denote any set of the form $[a,b[\times [c,d[$ where $a,b,c,d\in [0,1[$; $a<b$ and $c<d$. By the \textit{Euclidean closure} of such a rectangle we mean its closure  in the euclidean topology of $[0,1[^2$. Let $\Delta_2=\{(x,y)\in\mathbb{S}^2: x\leq y\}$ and let $\Delta$ be the diagonal of $\mathbb{S}$.
	
	For each $k\geq 1,$ let $L_k$ be the straight line joining the points $(0,\frac{1}{k})$ and $(1,1)$.

	\begin{proposition}
		[\cite{bbl}] {\label{rec}}
		There is a countable collection $\mathcal{T}$ of pairwise disjoint Sorgenfrey rectangles such that:
		
		\begin{itemize}
			\item [(1)] $\bigcup\T=\Delta_2\setminus\Delta$;
			\item [(2)] for each $T\in\T$, the Euclidean closure of $T$ is disjoint from $\Delta$;
			\item [(3)] for each $x\in [0,1[$ the set $\{T\in\T:T\cap(\{x\}\times ]x,1[)\neq\emptyset\}$ is infinite and can be indexed as $\{T_m:m\geq 1\}$ in such a way that for all $k$ points of $T_k$ lie above points of $T_{k+1}$.
			\item [(4)] for each $T\in\T$, there is a $k\geq 1$ such that $T$ is between $L_k$ and $L_{k+2}.$ 
		\end{itemize}
	\end{proposition}
	It worth pointing out that clause (4) is a consequence of the proof in \cite{bbl}.
	
	Let $T$ be an element of the partition $\mathcal{T}$. If $T=[a,b[\times [c,d[$, then define $T^{U}=[a,b[\times [\frac{c+d}{2},d[$, $T^{L}=[a,b[\times [c,\frac{c+d}{2}[$ and $T^{S}=[c,d[\times [a,b[$. Finally, let $\mathcal{T}^S=\{T^S: T\in\mathcal{T}\}.$
	
	\begin{rmk}
		Note that $T^{U}$ is the upper half of $T$, $T^{L}$ is the lower half of $T$ and $T^S$ is the reflection of $T$ across the diagonal $\Delta$.
	\end{rmk}
	
	\begin{theorem}
		The spaces $\Delta_2$ and $\mathbb{S}^2$ are homeomorphic.
	\end{theorem}
	
	\begin{proof}
		We shall define a homeomorphism $h:\Delta_2\rightarrow \mathbb{S}^2$. First, we will define the function $h$. After that we will prove that $h$ is 1-1 and onto. Finally, we will prove that $h$ and its inverse  are continuous.\paragraph{}
		For each $T=[a,b[\times [c,d[\in \T,$ we consider the following homeomomorphisms  $$h_{T^U,T^S}:T^U\to T^S \  {\rm and} \ h_{T^L,T}:T^L\to T$$ given by $h_{T^U,T^S}(x,y)= (2y-d,x )$ and $h_{T^L,T}(x,y)= (x,2y-c ),$ respectively. 
		Let $$h:=Id_\Delta\cup \bigcup\limits_{T\in \T}\left (h_{T^L,T}\cup h_{T^U,T^S} \right ), $$ where $Id_\Delta$ represents the identity function restricted to the diagonal.
		Since $\Delta_2=\Delta\sqcup \bigsqcup\limits_{T\in \T} (T^U\sqcup T^L)$ and $\mathbb{S}^2=\Delta\sqcup\bigsqcup\limits_{T\in \T} (T\sqcup T^S),$ it follows that  
		$h:\Delta_2\to \mathbb{S}^2$ is a bijection. Notice that $h\restriction (\Delta_2\setminus \Delta)$ and $h^{-1}\restriction (\mathbb{S}^2\setminus \Delta)$ are continuous,  as $h_{T^U,T^S}$ and  $h_{T^L,T}$ are homeomorphism between clopen subspaces. 
		
		We will now show that $h\restriction \Delta$ is continuous. Let $(x,x)\in\Delta$ and $ (x_n,y_n) (n\in\omega)$ be a sequence that converges to $(x,x)$. We may assume, without lose of generality, that $x_n>x,y_n>x$ for any $n\in\omega$. Separately, consider three subsequences, namely, those on the diagonal, those in the sets of the form $T^U$ and those in the sets of the form $T^L$. Let $A_0, A_1, A_2$ denote the indexes of the subsequences. If there are infinitely many points of the first type, then their images converges to $(x,x)$ since $h$ is the identity at such points. 
		
		If there are infinitely many points of the second type, then their images have the form $(2y_n-d_n,x_n ),$ where
		$T_n=[a_n,b_n[\times [c_n,d_n[$ is the element of $\T$ that contain $(x_n,y_n)$. Since $y_n$ converges to $x$ and $x<2y_n-d_n<y_n$, we have that  $\lim\limits_{n\in A_1}h(x_n,y_n)=\lim\limits_{n\in A_1}(2y_n-d_n,x_n)=(x,x)$.
		
		If there are infinitely many points of the third type, then their images have the form $ (x_n,2y_n-c_n ),$ where $T_n=[a_n,b_n[\times [c_n,d_n[$ is the element of $\T$ that contains $(x_n,y_n)$. Since $2y_n-c_n<d_n$, it is sufficient to prove that:
		
		\begin{claim}
			If $\lim\limits_{n\in A_2}(x_n,y_n)=(x,x)$, then $\lim\limits_{n\in A_2}(x_n,d_n)=(x,x)$.\\
		\end{claim}
		\begin{proof}
			Let $V=[x,x+\epsilon[^2\cap \Delta_2$ be a given open neighborhood of $(x,x)$.  Fix $k$ such that $\frac{1}{k}<\frac{\epsilon}{2}$. Since $(x_n,y_n)$ converges to $(x,x)$, we can find $N$ so that $(x_n,y_n)\in [x,x+\frac{\epsilon}{2}[^2$ and it is below the line $L_{k+2}$ for all $n\geq N$ . We are left  to show that $(x_n,d_n)\in V$ for all $n\geq N$. Let $n\geq N$ be given. Since $(x_n,y_n)$ is below the line $L_{k+2}$ and every rectangle $T\in \mathcal{T}$ is between two lines $L_{\ell} $ and $L_{\ell+2}$ for some $\ell.$ It follows that $T_n=[a_n,b_n[\times [c_n,d_n[$ is below $L_k$. Hence, $d_n-c_n<\frac{1}{k}<\frac{\epsilon}{2}$  and $x<y_n<x+\frac{\epsilon}{2}$. Therefore, $(x_n,d_n)\in V$ as required.
		\end{proof}
		
		We are left to show that the inverse $h^{-1}$ is continuous on the diagonal. Let $(x,x)\in\Delta$ be given and let $ (x_n,y_n) (n\in\omega)$ be a sequence that converges to $(x,x)$. We may assume, without lose of generality, that $x<x_n, y<y_n$ for any $n\in\omega$. Separately, consider three subsequences, namely, those on the diagonal, those elements above the diagonal and those in the sets below the diagonal. Let $A_0, A_1, A_2$ denote the indexes of the subsequences.
		If  there are infinitely many points of the first type, then their images converges to $(x,x)$ since $h^{-1}$ is the identity at such points.

		If there are infinitely many points of the second type, then their images have the form $(x_n,\frac{y_n+c_n}{2})$ where $T_n=[a_n,b_n[\times [c_n,d_n[$ are the elements of $\T$ that contains $(x_n,y_n)$. Since $y_n$ converges to $x$ and $x<\frac{y_n+c_n}{2}<y_n$, we have that $\lim\limits_{n\in A_1}(x_n,\frac{y_n+c_n}{2})=(x,x)$.  
		
		If there are infinitely many points of the third type, then their images have the form $(y_n,\frac{x_n+d_n}{2})$ where $T^S_n=[c_n,d_n[\times [a_n,b_n[$ are the elements of $\T$ such that $T_n^S$ contains $(x_n,y_n)$. Since $x<\frac{x_n+d_n}{2}
		<d_n$, it is sufficient to prove that:
		
		\begin{claim}
			If $\lim\limits_{n\in A_2}(x_n,y_n)=(x,x)$, then $\lim\limits_{n\in A_2}(y_n,d_n)=(x,x)$.\\
		\end{claim}
		\begin{proof}
			Let $V=[x,x+\epsilon[^2$ be a given open neighborhood of $(x,x)$. Let $\Tilde{L}_k$ denote the line from $(\frac{1}{k},0)$ to $(1,1)$ for $k\geq 1$.  Fix $k$ such that $\frac{1}{k}<\frac{\epsilon}{2}$. Since $(x_n,y_n)$ converges to $(x,x)$, we can find $N$ so that $(x_n,y_n)\in [x,x+\frac{\epsilon}{2}[^2$ and it is above the line $\Tilde{L}_{k+2}$ for all $n\geq N$ . We are left  to show that $(y_n,d_n)\in V$ for all $n\geq N$. Let $n\geq N$ be given. Since $(x_n,y_n)$ is above the line $\Tilde{L}_{k+2}$ and every rectangle $T\in \mathcal{T}^S$ is between two lines $\Tilde{L}_{\ell} $ and $\Tilde{L}_{\ell+2}$ for some $\ell.$ It follows that $T^S_n=[c_n,d_n[\times [a_n,b_n[$ is above $\Tilde{L}_k$. Hence, $d_n-c_n<\frac{1}{k}<\frac{\epsilon}{2}$  and $x<x_n<x+\frac{\epsilon}{2}$. Therefore, $(y_n,d_n)\in V$ as required.
		\end{proof}
	This concludes the proof of the Theorem.
	\end{proof}
	As a corollary we obtain.
	\begin{corollary}
		$\mathcal{F}_2(\mathbb{S})$ is homogeneous. 
	\end{corollary}
	
	\begin{proof}
		By Proposition \ref{rec}, $\mathcal{F}_2(\mathbb{S})$ is homeomorphic to $\Delta_2$. Since  $\mathbb{S}^2$ is homogeneous and $\Delta_2$ is homeomorphic to it, by the previous theorem, the result holds.
	\end{proof}

	
	


\addcontentsline{toc}{section}{References}
	\begin{thebibliography}{}
		
		
		
		
		
		
		
		
		
		
		\bibitem[Ar87]{ar87}Arkhangel'ski\v{i}, A. V. \emph{Topological homogeneity. Topological groups and their continuous images}. (Russian) Uspekhi Mat. Nauk 42 (1987), no. 2(254), 69–-105, 287. 
		\bibitem[AvM13]{avm}  Arhangel'ski\v{i}, A. V.; van Mill, J. \emph{On the cardinality of countable dense homogeneous spaces}. Proc. Amer. Math. Soc. 141 (2013), no. 11, 4031–-4038. 
		\bibitem[BBL12]{bbl} Bennett, Harold; Burke, Dennis; Lutzer, David \emph{Some questions of Arhangel'ski\v{i} on rotoids}. Fund. Math. 216 (2012), no. 2, 147–-161.
		\bibitem[En89]{en89}  Engelking, Ryszard \emph{General topology}. Translated from the Polish by the author. Second edition. Sigma Series in Pure Mathematics, 6. Heldermann Verlag, Berlin, 1989. viii+529 pp. ISBN: 3-88538-006-4 
		
		\bibitem[HG13]{hg}	Hernández-Gutiérrez, Rodrigo. \emph{  Countable dense homogeneity and the double arrow space}. Topology Appl. 160 (2013), no. 10, 1123–-1128.
		
		\bibitem[Ga54]{g54} Ganea, Tudor \emph{Symmetrische Potenzen topologischer Räume.} (German) Math. Nachr. 11 (1954), 305–316.
		
		
		
		
		
		\bibitem[SW75]{sw} Schori, R. M.; West, J. E. \emph{The hyperspace of the closed unit interval is a Hilbert cube.} Trans. Amer. Math. Soc. 213 (1975), 217–235.
		
		\bibitem[Sce76]{sce} Ščepin, E. V. \emph{Topology of limit spaces with uncountable inverse spectra. (Russian)} Uspehi Mat. Nauk 31 (1976), no. 5 (191), 191–226.
		
		
	\end{thebibliography}
		
		
		
		


	\bigskip
	\noindent Sebastián Barr\'ia\\
	Email: sebabarria@udec.cl\\
	
	\noindent Carlos A. Mart\'inez-Ranero\\
	Email: cmartinezr@udec.cl\\
	Homepage: www2.udec.cl/~cmartinezr\\
	
	\noindent Same Address:\\ 
	Universidad de Concepci\'on, Concepci\'on, Chile\\
	Facultad de Ciencias F\'isicas y Matem\'aticas\\
	Departamento de Matem\'atica\\


\end{document}